\documentclass[12pt]{amsart}

\usepackage{amscd, amsfonts, amsmath, amssymb, amsthm, mathrsfs, appendix, enumerate, graphicx, flexisym, epsfig, float, graphicx, hyperref, pdfpages, tikz, mathtools, comment, tikz-cd, tabu}
\usepackage[all,cmtip]{xy}
\usepackage{graphicx}

\makeatletter
\@namedef{subjclassname@2020}{Mathematics Subject Classification \textup{2020}}

\newsavebox{\@brx}
\newcommand{\llangle}[1][]{\savebox{\@brx}{\(\m@th{#1\langle}\)}%
  \mathopen{\copy\@brx\kern-0.5\wd\@brx\usebox{\@brx}}}
\newcommand{\rrangle}[1][]{\savebox{\@brx}{\(\m@th{#1\rangle}\)}%
  \mathclose{\copy\@brx\kern-0.5\wd\@brx\usebox{\@brx}}}

\newtheorem{theorem}{Theorem}[section]
\newtheorem{corollary}[theorem]{Corollary}
\newtheorem{lemma}[theorem]{Lemma}
\newtheorem{proposition}[theorem]{Proposition}
\theoremstyle{definition}

\newtheorem{problem}[theorem]{Problem}

\newtheorem{remark}[theorem]{Remark}

\numberwithin{equation}{subsection}
\newtheorem*{ack}{Acknowledgement}
\usepackage[all,cmtip]{xy}

\newcommand{\Aut}{\operatorname{Aut}}

\newcommand{\im}{\operatorname{im}}

\newcommand{\Oo}{\operatorname{O}}

\newcommand{\GL}{\operatorname{GL}}
\newcommand{\SL}{\operatorname{SL}}

\newcommand{\id}{\mathrm{id}}

\begin{document}

\title[Congruence subgroups and crystallographic quotients]{Congruence subgroups and crystallographic quotients of small Coxeter groups}

\author{Pravin Kumar}
\email{pravin444enaj@gmail.com}
\address{Department of Mathematical Sciences, Indian Institute of Science Education and Research (IISER) Mohali, Sector 81,  S. A. S. Nagar, P. O. Manauli, Punjab 140306, India.}

\author{Tushar Kanta Naik}
\email{tushar@niser.ac.in}
\address{School of Mathematical Sciences, National Institute of Science Education and Research, Bhubaneswar, An OCC of Homi Bhabha National Institute, P. O. Jatni, Khurda 752050, Odisha, India.}
\author{Mahender Singh}
\email{mahender@iisermohali.ac.in}
\address{Department of Mathematical Sciences, Indian Institute of Science Education and Research (IISER) Mohali, Sector 81,  S. A. S. Nagar, P. O. Manauli, Punjab 140306, India.}


\begin{abstract}
Small Coxeter groups are precisely the ones for which the Tits representation is integral, which makes the study of their congruence subgroups relevant.  The symmetric group $S_n$ has three natural extensions, namely, the braid group $B_n$, the twin group $T_n$ and the triplet group $L_n$. The latter two groups are small Coxeter groups, and play the role of braid groups under the Alexander-Markov correspondence for  appropriate knot theories, with their pure subgroups admitting suitable hyperplane arrangements as Eilenberg-MacLane spaces.  In this paper, we prove that the congruence subgroup property fails for infinite small Coxeter groups which are not virtually abelian. As an application, we deduce that the congruence subgroup property fails for both $T_n$ and $L_n$ when $n \ge 4$. We also determine subquotients of principal congruence subgroups of $T_n$, and identify the pure twin group $PT_n$ and the pure triplet group $PL_n$ with suitable principal congruence subgroups. Further, we investigate crystallographic quotients of these two families of small Coxeter groups, and prove that $T_n /PT_n^{'}$, $T_n/T_n^{''}$ and $L_n /PL_n^{'}$ are crystallographic groups. We also determine crystallographic dimensions of these groups and identify the holonomy representation of $T_n/T_n^{''}$.

\end{abstract}

\maketitle
% \tableofcontents

\section{Introduction}
A group $W$  is called a {\it Coxeter group} if it admits a  presentation of the form $\big \langle S \mid R \big \rangle$, where $S= \{ w_i \mid i \in \Pi\}$ is the set of generators and 
\begin{equation}\label{Coxeter presentation}
R  = \big \{ (w_iw_j)^{m_{i,j}}~ \mid m_{i,j}\in \mathbb{N} \cup \{\infty\},~m_{i,j}=m_{j,i}, ~m_{i,j} = 1 ~\textrm{if and only if}~ i =j \big\}
\end{equation}
is the set of defining relations. The pair $(W, S)$ is called a  {\it Coxeter system}. We refer the $m_{i,j}$'s as {\it exponents} of the Coxeter system $(W, S)$. The cardinality of the set $S$ is called the {\it rank} of the Coxeter system, and we shall consider only finite rank Coxeter systems.
\par

For convenience of terminology, we say that a group is a {\it small Coxeter group} if it admits a Coxeter system such that each exponent $m_{i, j}$ is either $\infty$ or less than or equal to 3. For instance, symmetric groups, right-angled Coxeter groups and universal Coxeter groups are small Coxeter groups. The main objects of this paper are two classes of small Coxeter groups that arise as natural extensions of symmetric groups. For  $n \ge 2$, the  symmetric group $S_n$ has a Coxeter presentation with generators $\{\tau_1, \ldots, \tau_{n-1} \} $ and the following defining relations:
\begin{enumerate}
\item $\tau_i^2 = 1$ for  $1\leq i \leq n-1$.
\item $\tau_i\tau_{i+1}\tau_i=\tau_{i+1}\tau_i\tau_{i+1}$ for $1\leq i \leq n-2$.
\item $\tau_i\tau_j=\tau_j\tau_i$ for $\mid i - j\mid \geq 2$.
\end{enumerate}
\par

By omitting all relations of type (1), (2) or (3) one at a time from the preceding presentation, we obtain presentations of the braid group $B_n$, the twin group $T_n$ and the triplet group  $L_n$ as follows:
\begin{small}
\begin{align}\label{definition-Ln}
\nonumber B_n &= \big{\langle} \sigma_1, \sigma_2, \ldots, \sigma_{n-1} ~\mid~ \sigma_j\sigma_k=\sigma_k\sigma_j,~~ \sigma_i\sigma_{i+1}\sigma_i=\sigma_{i+1}\sigma_i\sigma_{i+1}~~\text{for all} ~~\mid j - k\mid \geq 2~~\text{and} ~~1\leq i \leq n-2 \big{\rangle}.&\\
\nonumber T_n &= \big{\langle} s_1, s_2, \ldots, s_{n-1} ~\mid~ s_k^2 = 1,~~s_is_j=s_js_i~~ \text{for all}~~1\leq k \leq n-1~~\text{and} ~~  \mid i - j\mid \geq 2 \big{\rangle}.&\\
\nonumber L_n &= \big{\langle} y_1, y_2, \ldots, y_{n-1} ~\mid~ y_j^2 = 1,~~y_iy_{i+1}y_i=y_{i+1}y_i y_{i+1} ~\text{for all} ~~1\leq j \leq n-1~~\text{and} ~~1\leq i \leq n-2 \big{\rangle}.&
\end{align}
\end{small}

The terms twin and triplet groups were coined by Khovanov \cite{MR1386845} in his study of $K(\pi,1)$ subspace arrangements. It has been proved that the pure subgroups $PT_n$ and $PL_n$, which are kernels of natural surjections $T_n \to S_n$ and $L_n \to S_n$ respectively, have suitable hyperplane arrangements as their Eilenberg--MacLane spaces. Furthermore, twin and triplet groups have topological interpretations and can be thought of as real forms of braid groups. These groups play the role of braid groups under the Alexander-Markov correspondence for some appropriate knot theories.  It turns out that twin groups are related to doodles on the 2-sphere, which were introduced by Fenn and Taylor \cite{MR0547452} as finite collections of simple closed curves on the 2-sphere without triple or higher intersections. Allowing self-intersections of curves, Khovanov \cite{MR1370644} extended the idea to finite collections of closed curves without triple or higher intersections on a closed oriented surface. Also, Khovanov \cite{MR1386845} associated a group to a doodle, which can be thought of as an analogue of the fundamental group of the link complement of a link in the 3-space. Similarly, triplet groups are related to topological objects called noodles. Fixing a codimension one foliation (with singular points) on the 2-sphere, a noodle is a collection of closed curves on the 2-sphere such that no two intersection points belong to the same leaf of the foliation, there are no quadruple intersections and no intersection point can occupy a singular point of the foliation. 
\par

For $n \ge 2$, an element of the twin group $T_n$ can be identified with the homotopy class of a configuration of $n$ arcs in the infinite strip $\mathbb{R} \times  [0,1]$ connecting fixed $n$ marked points on each of the parallel boundary lines such that each arc is monotonic and no three arcs have a point in common. Taking the one-point compactification of the plane, one can define the closure of a twin on a $2$-sphere analogous to the closure of a geometric braid in the 3-space. While Khovanov proved that every oriented doodle on a $2$-sphere is the closure of a twin, an analogue of the Markov Theorem in this setting has been proved by Gotin \cite{Gotin}, though the idea has been implicit in \cite{Khovanov1990}. Automorphisms, conjugacy classes and centralisers of involutions in twin groups have been explored in \cite{MR4145210, MR4192499}. A similar study for a family of odd Coxeter groups that includes triplet groups has been carried out in \cite{MR4270786}. It has been proved in \cite{MR4027588} that pure twin groups $PT_3$ and $PT_4$ are free groups of ranks 1 and 7, respectively. Gonz\'alez, Le\'on-Medina and Roque \cite{MR4170471} have shown that $PT_5$ is a free group of rank $31$. Further, Mostovoy and Roque-M\'arquez \cite{MR4079623} have proved that $PT_6 \cong F_{71} *_{20} \big(\mathbb{Z} \oplus \mathbb{Z} \big)$. Recently, a minimal presentation of $PT_n$ for all $n$ has been announced by Mostovoy \cite{Mostovoy}, though the precise structure of this group still remains mysterious.  In a recent preprint \cite{Farley2021}, Farley has shown that $PT_n$ is always a diagram group. It is worth noting that (pure) twin groups have been used by physicists in the study of three-body interactions and topological exchange statistics in one dimension  \cite{MR4035955, MR4440997}. Concerning triplet groups, it has been shown in \cite{MR4270786} that the pure triplet group $PL_n$ is a non-abelian free group of finite rank for $n \ge 4$.
\par

It is a well-known result of Tits that for a Coxeter system $(W, S)$ of rank $n$, the group $W$ admits a canonical faithful representation $W \to \GL(n, \mathbb{R})$.  The paper begins with the observation that the Tits representation of a small Coxeter group is integral. This makes the study of congruence subgroups of such groups relevant, which we pursue in sections \ref{sec2} and \ref{sec3}. We prove that if $W$ is a right-angled Coxeter group, then $W[4] \cong W^{'}$ (Proposition \ref{racg commutator}). In a general setting, we prove that an infinite small Coxeter group which is not virtually abelian does not have the congruence subgroup property (Theorem \ref{virtual abelian and csp}). As an application, we deduce that the congruence subgroup property fails for both $T_n$ and $L_n$ when $n\geq 4$ (Propositions \ref{csp fails for T_n} and \ref{csp fails for L_n}). Further, we investigate subquotients of  principal congruence subgroups of $T_n$, and prove that if $3 \nmid m$, then $T_n[m]/T_n[3m]$ is isomorphic to the alternating group $A_n$ (Theorem \ref{quotient twin cong subgroup}). We also identify pure twin and pure triplet groups with principal congruence subgroups. More precisely, we prove that $T_n[3]=T_n[6] =PT_n$ (Proposition \ref{tn3=ptn}) and $L_n[2]=PL_n$ (Proposition \ref{ln2=pln}). It is known that crystallographic groups play an important role in the study of groups of isometries of Euclidean spaces. In Section \ref{sec4}, we determine some natural crystallographic quotients of twin groups and triplet groups. We prove that $T_n /PT_n^{'}$ is a crystallographic group of dimension $2^{n-3}(n^2-5n+8)-1$ for $n\geq 4$ (Theorem \ref{tn mod ptn cryst}) and $T_n/T_n^{''}$ is a crystallographic group of dimension $2n-5$ for $n\geq 3$ (Theorem \ref{tn mod tn' cryst}). Finally, concerning triplet groups, we prove that $L_n /PL_n^{'}$ is a crystallographic group of dimension $1+ n!\Large(\frac{2n-7}{6}\Large)$ (Theorem \ref{ln mod pln cryst}). 
\medskip

\section{Tits representation of Coxeter groups}\label{sec2}
Let $W$ be a Coxeter group given by a Coxeter presentation $W= \langle S \mid R \rangle $, where $S= \{ w_i \mid i \in \Pi\}$. Let $V$ be the real vector space spanned by the set $\{e_i \mid i \in \Pi \}$. Define a symmetric bilinear form $B$ on $V$ by
$$
B\left(e_{i}, e_{j}\right)= \begin{cases}-\cos \left(\frac{\pi}{m_{i, j}}\right) & \text { if } m_{i, j} \neq \infty, \\ -1 & \text { if } m_{i, j}=\infty. \end{cases}
$$
Then, for each $i \in \Pi$, the linear map $\sigma_{i}: V \rightarrow V$ given by 
\begin{equation}\label{Tits representation formula}
\sigma_{i}(v)=v-2 B\left(e_{i}, v\right) e_{i},
\end{equation}
 defines an automorphism of $V$. The following result is folklore \cite[Chapter V, Section 4]{Bourbaki}.

\begin{theorem}
The map $\rho: W \rightarrow \GL(V)$ defined through $\rho\left(w_{i}\right)=\sigma_{i}$ is a faithful representation of $W$.
\end{theorem}

The  representation $\rho: W \rightarrow \GL(V)$ is called the Tits representation of $W$. The following is an immediate observation for small Coxeter groups.

\begin{lemma}\label{tits rep integral small}
The Tits representation of a Coxeter group is integral if and only if it is a small Coxeter group.
\end{lemma}

\begin{proof}
Let $W= \langle S \mid R \rangle $ be a Coxeter group and $\rho: W \rightarrow \GL(V)$ given by $\rho(w_i)=\sigma_i$ be its Tits representation. Clearly, the entries of the matrix of $\sigma_i$ lie in $\{0, 1, -1, -2B(e_i,e_j) \}$. It follows that $-2B(e_i,e_j)= 2 \cos(\frac{\pi}{m_{i,j}})$ is an integer if and only if  $m_{i,j}= 1,2,3$ or ${\infty}$.
\end{proof}

\subsection*{Tits representation of twin groups} For an integer $n \ge 2$, the twin group $T_n$ is a Coxeter group with the presentation
$$T_n = \big{\langle} s_1, s_2, \ldots, s_{n-1} ~\mid~ s_k^2 = 1,~~s_is_j=s_js_i~~ \text{for all}~~1\leq k \leq n-1~~\text{and} ~~  \mid i - j\mid \geq 2 \big{\rangle}.$$
It follows that $T_2 \cong \mathbb{Z}_2$ and $T_3 \cong \mathbb{Z}_2 *\mathbb{Z}_2$, the infinite dihedral group. Notice that twin groups form a special class of right-angled Coxeter groups. Let $I_m$ denote the $m\times m$ identity matrix. Then the Tits representation $\rho: T_n \to \GL(n-1, \mathbb{Z})$ is given by $\rho(s_i)= X_i$, where  
\begin{equation*}
X_i = 
\begin{cases}
\left(\begin{array}{cc|c}
-1 & 2 & 0 \\
0 & 1 & \\
\hline 0 & 0 & I_{n-3}
\end{array}\right) & \text{if}\; i= 1,\\
\\
\left(\begin{array}{c|ccc|c}
I_{i-2} & 0 & 0 & 0 & 0 \\ \hline & 1 & 0 & 0 &\\  0 & 2 & -1 & 2 & 0 \\ & 0 & 0 & 1 & \\ \hline 0 & 0 & 0 & 0 & I_{n-(i+2)}\end{array}\right) & \text{if}\;  2\leq i \leq n-2,\\
\\
\left(\begin{array}{c|c}
I_{n-3} & 0 \\
\hline
0 & \begin{array}{cc}
1 & 0 \\
2 & -1
\end{array}
\end{array}\right) & \text{if}\; i= n-1.
\end{cases}
\end{equation*}

\subsection*{Tits representation of triplet groups} For $n \ge 2$, the triplet group $L_n$ is a Coxeter group with the presentation
$$ L_n = \big{\langle} y_1, y_2, \ldots, y_{n-1} ~\mid~ y_j^2 = 1,~~y_iy_{i+1}y_i=y_{i+1}y_i y_{i+1} ~\text{for all} ~~1\leq j \leq n-1~~\text{and} ~~1\leq i \leq n-2 \big{\rangle}.$$
It follows that $L_2 \cong \mathbb{Z}_2$ and $L_3 \cong S_3$. Let $M_l$ stand for the $3\times l$ matrix with zeroes on the first and the third rows, and 2's as its middle row entries. For example, 
$$
M_4= \left(\begin{matrix}
0 & 0 & 0 & 0\\
2 & 2 & 2 & 2\\
0 & 0 & 0 & 0\\
\end{matrix}\right).
$$
Let $N_l$ be the $2\times l$ matrix consisting of the second and the third row of $M_l$. Similarly, let $K_l$ be the $2\times l$ matrix consisting of the first and the second row of $M_l$. Then the Tits representation $\rho: L_n \to \GL(n-1, \mathbb{Z})$ is given by $\rho(y_i)=Y_i$, where  
\begin{equation*}
Y_i = 
\begin{cases}
\left(\begin{array}{cc|c}
-1 & 1 & N_{n-3} \\
0 & 1 & \\
\hline 0 & & I_{n-3}
\end{array}\right) & \text{if}\; i= 1,\\
\\
\left(\begin{array}{c|ccc|c}I_{i-2} &  & 0 &  & 0 \\ \hline & 1 & 0 & 0 &\\ M_{i-2} & 1 & -1 & 1 & M_{n-(i+2)} \\ & 0 & 0 & 1 & \\ \hline 0 &  & 0 &  & I_{n-(i+2)}\end{array}\right) & \text{if}\;  2\leq i \leq n-2,\\
\\
\left(\begin{array}{c|c}
I_{n-3} & 0 \\
\hline
K_{n-3} & \begin{array}{cc}
1 & 0 \\
1 & -1
\end{array}
\end{array}\right) & \text{if}\; i= n-1.
\end{cases}
\end{equation*}
\medskip

\section{Congruence subgroups of small Coxeter groups}\label{sec3}
Given a representation $G \to \mathrm{GL}(n, \mathbb{Z})$ of a group $G$ and an integer $m \ge 2$, one defines the {\it principal congruence subgroup} $G[m]$ of level $m$ as the kernel of the composition $$G \to \GL(n, \mathbb{Z}) \rightarrow \GL(n, \mathbb{Z}_m).$$ Notice that $G[m] \le G[k]$ for each divisor $k$ of $m$. We say that the group $G$ has the {\it congruence subgroup property} if every finite index subgroup of $G$ contains some principal congruence subgroup. A finite index subgroup of $G$ containing some principal congruence subgroup is called a {\it congruence subgroup}.  In the case of braid groups (and mapping class groups), one uses the usual symplectic representation to define principal congruence subgroups. We refer the reader to \cite{quotientbraid, BPS2022, MR4157115, MR3757477, MR3786423} for recent results and a survey of congruence subgroups of braid groups.
\par

In view of Lemma \ref{tits rep integral small}, the Tits representation of a small Coxeter group is integral, and hence it is interesting to explore its (principal) congruence subgroups. Let $W= \langle S \mid R \rangle$ be a small Coxeter group with $|S|=n-1$, where $n \ge 3$. For each  $m \ge 2$, let  $$\rho_m: W \to  \GL(n-1,\mathbb Z_m)$$ be the composition of $\rho: W \to  \GL(n-1,\mathbb Z)$ with the modulo $m$ reduction homomorphism $\GL(n-1,\mathbb Z) \to \GL(n-1,\mathbb Z_m)$. Let $W[m]:=\ker(\rho_m)$ denote the principal congruence subgroup of $W$ of level $m$. We note that the definition depends on our Coxeter system $(W, S)$.
\par

We begin with the following basic observation. 

\begin{lemma}\label{length}
Let $W= \langle S \mid R \rangle$ be a small Coxeter group with $|S| \ge 2$ and $m \ge 3$. Then each element of $W[m]$ is a word of even length in $W$.
\end{lemma}

\begin{proof}
Notice that $\det(\rho(w_i))=-1$ for all $w_i \in S$. For a reduced word $w=w_{i_1} w_{i_2}\cdots w_{i_k} \in W$, if $\ell(w)=k$ denotes the length of $w$, then $\det(\rho(w))=(-1)^{\ell(w)}$. If $w \in W[m]$, then $\det(\rho_m(w))=1 \mod m$, that is, $(-1)^{\ell(w)}=1 \mod m$. It follows that $\ell(w)$ is an even integer, and hence $w$ is a word of even length.
\end{proof}

Obviously, the converse of the preceding lemma is not true. For example, if we consider $\rho: T_n \to \GL(n-1, \mathbb{Z})$, then 
$$\rho(s_1s_2)=\left(\begin{array}{ccc|c}
3 & -2 & 4 & 0 \\
2 & -1 & 2 & 0 \\
0 & 0 & 1 & 0 \\
\hline
0 & 0 & 0 & I_{n-4}
\end{array}\right)$$
is not identity modulo any $m\geq 3$.

\begin{corollary}\label{prin cong torsion free}
Let $W= \langle S \mid R \rangle$ be a small Coxeter group with $|S| \ge 2$. Then the level $m$ principal congruence subgroup of $W$ is torsion free for each $m\geq 3$. In particular, if $W$ is finite, then  the level $m$ principal congruence subgroup of $W$ is trivial for each $m\geq 3$.
\end{corollary}

\begin{proof}
Suppose that $|S|=n-1$, where $n \ge 3$. The Tits representation $\rho: W \to \GL(n-1, \mathbb Z)$ when restricted to the alternating subgroup $E(W)$ of all words of even lengths can be thought of as a representation $E(W)\to \SL(n-1,\mathbb Z)$. Note that $$E(W)[m]=\left(\rho{|_{E(W)}}\right)^{-1}\left(\SL(n-1,\mathbb Z)[m]\cap \rho(E(W))\right).$$ 
Since the principal congruence subgroup $\SL(n-1,\mathbb Z)[m]$ of $\SL(n-1,\mathbb Z)$ is torsion free for $m\geq 3$ \cite[Proposition 6.7]{MR2850125}, it follows that $W[m]$ is torsion free for each $m\geq 3$. This completes the proof.
\end{proof}

\begin{corollary}
A finite small Coxeter group has the congruence subgroup property.
\end{corollary}
\medskip

Let $W= \langle S \mid R \rangle$ be a small Coxeter group with $S=\{w_1,w_2,\ldots,w_{n-1}\}$, where $n \ge 3$. Fixing the ordered basis $\{e_1,e_2,\ldots,e_{n-1}\}$ for the real vector space $V$, by Lemma \ref{tits rep integral small}, we can view the Tits representation as the matrix representation $\rho:W \to \GL(n-1,\mathbb Z)$. For $1 \le k \le n-1$, if $(a_{i,j})$ denotes the matrix of $\sigma_k$ (see \eqref{Tits representation formula}), then one can see that
\begin{equation}\label{small coxeter matrix}
a_{i,j}=\begin{cases}
\delta_{i,j} & \text{ if } i\neq k,\\
\alpha(k,j) & \text{ if } i = k,
\end{cases}
\end{equation}
where $\alpha(k,j)=\alpha(j,k)=\begin{cases}
1 & \text{ if } m_{k,j} = 3, \\
0 & \text{ if } m_{k,j} = 2, \\
2 & \text{ if } m_{k,j} = \infty, \\
-1 & \text{ if } j=k.
\end{cases}
$

Now, if $1 \le k\neq \ell \le n-1$ and $(c_{i,j})$ denotes the matrix of $\sigma_k\sigma_l$, then 
\begin{equation}\label{cij}
c_{i,j}=\begin{cases}
    \delta_{i,j} & \text{ if } i\neq k,\quad i\neq \ell,\\
    \alpha(\ell,j) & \text{ if } i\neq k,\quad i=\ell,\\
    -\alpha(k,\ell) & \text{ if } i=k,\quad j=\ell,\\
	\alpha(k,j)+ \alpha(\ell,j)\alpha(k,\ell)    & \text{ if } i=k, \quad j\neq \ell.
\end{cases}
\end{equation}

Further, note that
$$c_{k,k}=\alpha(\ell,k)^2-1,\quad c_{\ell,\ell}=-1,\quad c_{k,\ell}=-\alpha(k,\ell) \quad \textrm{and} \quad c_{\ell,k}=\alpha(\ell,k).$$

For $j\neq \ell$, we set $$\gamma_j :=c_{k,j}=\begin{cases}
        \alpha(\ell,k)^2 -1 & \text{ if } j= k,  \\
       	\alpha(k,j)+ \alpha(\ell,j)\alpha(k,\ell) & \text{ if } j\neq k.
    \end{cases}$$

If $W$  is right-angled, then
\begin{equation}\label{econ}
\alpha(i,j)=\begin{cases}
~~0\mod 2 &\text{if } i\neq j,\\
-1\mod 2&\text{if } i = j,
    \end{cases} \quad \textrm{and} \quad \gamma_j=\begin{cases} -1\mod 4 & \text{if } j=k,\\
    ~~0 \mod 2 &\text{if } j\neq k.
    \end{cases}
\end{equation}

\begin{lemma}\label{prod matrix square}
Let $1 \le k\neq \ell \le n-1$  and let $(d_{i,j})$ be the matrix of $\left(\sigma_k\sigma_l\right)^2$. Then 
$$d_{i,j}=\delta_{i,j} \quad \textrm{if} \quad i\neq k,\ell, $$
$$
d_{k,j}=\begin{cases}
      \alpha(k,\ell)^4-3\alpha(k,\ell)^2+1 & \text{ if } j=k,\\
      -\alpha(k,\ell)^3+2\alpha(k,\ell) & \text{ if }  j=\ell,\\
	\gamma_j\alpha(k,\ell)^2-\alpha(k,\ell)\alpha(\ell,j)  & \text{ if }  j\neq k ,\ell,
\end{cases}
\quad \textrm{and} \quad 
d_{\ell,j}=\begin{cases}
        \alpha(\ell,k)^3-2\alpha(\ell,k)  &\text { if } j=k,\\
		-\alpha(\ell,k)^2 + 1  &\text { if } j=\ell, \\
		\alpha(\ell,k)\gamma_j	 &\text { if } j\neq k, \ell.
      \end{cases}
$$
\end{lemma}

\begin{proof}
If $i\neq k, \ell$, then 
$$    d_{i,j} =\sum_{p=1}^{n-1} c_{i,p}c_{p,j}   =\sum_{p=1}^{n-1} \delta_{i,p}c_{p,j} =c_{i,j} = \delta_{i,j}.$$ 
If $i=k$, then 
$$
\begin{aligned}
  d_{k,j}  &=\sum_{p=1}^{n-1} c_{k,p}c_{p,j} \\
      &=\sum_{\substack{p=1\\p\neq k,\ell}}^{n-1}c_{k,p}c_{p,j} + c_{k,k}c_{k,j} + c_{k,\ell}c_{\ell,j}\\
      &=\sum_{\substack{p=1\\p\neq k,\ell}}^{n-1}c_{k,p}\delta_{p,j} + c_{k,k}c_{k,j} + c_{k,\ell}c_{\ell,j}\\
      &=\begin{cases}
          c_{k,k}^2 + c_{k,\ell}c_{\ell,k} &\text { if } j=k,\\
          c_{k,\ell}(c_{k,k} - 1) &\text { if } j=\ell, \\
          c_{k,j}(1+c_{k,k})+c_{k,\ell}c_{\ell,j} &\text { if } j\neq k, \ell,
      \end{cases} \\
      &=\begin{cases}
      \alpha(k,\ell)^4-3\alpha(k,\ell)^2+1 & \text{ if } j=k,\\
      -\alpha(k,\ell)^3+2\alpha(k,\ell) & \text{ if }  j=\ell,\\
	\gamma_j\alpha(k,\ell)^2-\alpha(k,\ell)\alpha(\ell,j)  & \text{ if }  j\neq k ,\ell.
\end{cases}
\end{aligned}
$$
If $i=\ell$, then 
$$
\begin{aligned}
  d_{\ell,j}  &=\sum_{p=1}^{n-1} c_{\ell,p}c_{p,j} \\
      &=\sum_{\substack{p=1\\p\neq k,\ell}}^{n-1}c_{\ell,p}\delta_{p,j} + c_{\ell,k}c_{k,j} + c_{\ell,\ell}c_{\ell,j}\\
      &=\begin{cases}
         c_{\ell,k}(c_{k,k} - 1)  &\text { if } j=k,\\
		c_{\ell,k}c_{k,\ell} + 1  &\text { if } j=\ell, \\
         c_{\ell,j} + c_{\ell,k}c_{k,j} -c_{\ell,j} &\text { if } j\neq k, \ell,
      \end{cases} \\
      &=\begin{cases}
        \alpha(\ell,k)^3-2\alpha(\ell,k)  &\text { if } j=k,\\
		-\alpha(\ell,k)^2 + 1  &\text { if } j=\ell, \\
		\alpha(\ell,k)\gamma_j	 &\text { if } j\neq k, \ell.
      \end{cases}
      \end{aligned}
$$
This concludes the proof.
\end{proof}

Throughout, we denote the commutator subgroup of a group $G$ by $G^{'}$.

\begin{proposition}\label{racg commutator}
If $W$ is a right-angled Coxeter group, then $W[2]= W$ and $W[4]= W^{'}$.
\end{proposition}
\begin{proof}
The assertion $W[2]= W$ is immediate from \eqref{small coxeter matrix}. Let $W= \langle S \mid R \rangle$ be a right-angled Coxeter group with $|S|=n-1$, and let $q:W\to \mathbb Z_2^{n-1}$ be the abelianisation map. Equation \eqref{econ} and Lemma \ref{prod matrix square} imply that the matrix $(d_{i,j})$ is the identity modulo 4. Thus, the map $\psi:\mathbb Z_2^{n-1} \to \GL(n-1,\mathbb Z_4)$, given by $\psi\left(q(w_i)\right)=\rho_4(w_i)$, is a group homomorphism.  We claim that $\psi$ is injective. Let $u_1,u_2,\ldots,u_{n-1}$ be generators for $\mathbb Z_2^{n-1}$, where $u_i=q(w_i)$.  Let $w=u_{j_1}u_{j_2}\cdots u_{j_r} \in  \mathbb Z_2^{n-1}$ be a word of length $r \ge 1$, where $j_1<j_2<\cdots < j_r$. Using \eqref{cij} and induction on $r$, one can see that if $\ell>j_r$, then the $\ell$-th row of $\psi(w)$ has $(\ell,\ell)$-entry as 1 and all other entries as 0. Now, suppose that  $u=u_{i_1}u_{i_2}\cdots u_{i_k} \in \ker(\psi)$ for some $k\ge1$. Since $u_i \not\in \ker(\psi)$ for all $i$, we can assume that $k>1$ and $i_1<i_2<\cdots < i_k$. Note that the $i_k$-th row of $\psi(u_{i_1}u_{i_2}\cdots u_{i_{k-1}})$ has $(i_k,i_k)$-entry as 1 and all other entries are 0, whereas the $i_k$-th column of $\psi(u_{i_k})$ has $(i_k,i_k)$-entry as $-1$ and all other entries are 0. Thus, $\psi(u)$ has $(i_k,i_k)$-entry as $-1$, which is a contradiction. Hence, $\psi$ is injective, and we obtain $W[4]=\ker (\rho_4)=\ker(q)=W^{'}$.
\end{proof}

\medskip

\subsection{Congruence subgroup property for small Coxeter groups}\label{sec congruence subgroup property}
In this section, we explore which small Coxeter groups have the congruence subgroup property.  Let $\mathcal{P}$ be a property of groups. Then a group is said to be virtually $\mathcal{P}$ if it has a finite index subgroup with property $\mathcal{P}$. We need the following well-known facts to prove the main results of this section.

\begin{proposition}\cite[Theorem II]{MR2700693}\label{propF2}
An infinite Coxeter group which is not virtually abelian has a finite index subgroup which surjects onto a free group of rank two.
\end{proposition}

\begin{proposition}\cite[Theorem 3.8.1]{surybook}\label{necesarycsp}
Let $G$ be a subgroup of $\SL(n,\mathbb Z)$ which has the congruence subgroup property. Then there is no finite index subgroup of $G$ which surjects onto a non-abelian free group. In particular, $G$ is not virtually free.
\end{proposition}

It is remarked in  \cite[p.312]{MR1626421} that for Coxeter groups which are not virtually abelian, the existence of a homomorphism onto the infinite cyclic group suggests that, in some sense, this class of Coxeter groups cannot have the congruence subgroup property. We confirm this remark for small Coxeter groups.

\begin{theorem}\label{virtual abelian and csp}
An infinite small Coxeter group which is not virtually abelian does not have the congruence subgroup property.
\end{theorem}

\begin{proof}
Let $W$ be an infinite small Coxeter group which is not virtually abelian. Note that the subgroup $E(W)$ of all even length words of $W$ is of index two in W, and the Tits representation $\rho$ restricts to a representation $E(W) \to \SL(n,\mathbb Z)$. Since $W$ is not virtually abelian, by Proposition \ref{propF2}, there exists a finite index subgroup $H$ of $W$ and a surjective group homomorphism $\phi:H \twoheadrightarrow F_2.$ If we set $K := H\cap E(W)$, then either $K=H$ or $K$ is an index two subgroup of $H$. It is a basic observation that the index $[\phi(H): \phi(K)]$ divides the index $[H: K]$. Thus,  $\phi(K) $ is either $F_2$ or a free subgroup of rank 3 (by the Nielsen-Schreier Theorem). Consequently, $\rho(K)$ is a finite index subgroup of $\rho(E(W))$ which surjects onto a non-abelian free group. By Proposition \ref{necesarycsp}, there exists a finite index subgroup $P$ of $\rho(E(W))$ which does not contain any principal congruence subgroup $\rho(E(W)) \cap \SL(n,\mathbb Z)[m]$ of $\rho(E(W))$. Note that $\rho^{-1}(P)$ is also a finite index subgroup of $W$ and it does not contain any principal congruence subgroup $W[m]=\rho^{-1}(\rho(W) \cap \GL(n,\mathbb Z)[m])$ of $W$, which proves the theorem.
\end{proof}
\par

Given a simple graph $\Gamma$ with vertex set $V(\Gamma)$ and edge set $E(\Gamma)$, we can define a right-angled Coxeter group $W(\Gamma)$ by the presentation 
$$
W(\Gamma)=\langle V(\Gamma)~ \mid~ v_i^2=1~ \textrm{for all}~v_i \in V(\Gamma) ~\textrm{and}~ v_i v_j= v_j v_i~\textrm{whenever}~ (v_i, v_j)  \in E(\Gamma) \rangle.
$$
Recall that, the join of two graphs $\Gamma_1$ and $\Gamma_2$ is a graph constructed from the disjoint union of $\Gamma_1$ and $\Gamma_2$ by connecting each vertex of $\Gamma_1$ to each vertex of $\Gamma_2$. The following fact is well-known, but we give a proof here for the sake of completeness.

\begin{lemma}\label{virtual abelian coxeter group criteria}
 Let $\Gamma$ be a finite simple graph. Then the right-angled Coxeter group $W(\Gamma)$ has a finite index abelian subgroup of rank $n\geq 0$  if and only if $\Gamma$  is the join of a complete graph on $m$ vertices for some $m\geq 0$ and $n$ copies of a graph with two isolated vertices.
\end{lemma}

\begin{proof}
Suppose that the right-angled Coxeter group $W(\Gamma)$ has a finite index abelian subgroup of rank $n\geq 0$. If $|V(\Gamma)| \leq 3$, then $W(\Gamma)$ is one of the following  groups: $\mathbb Z_2$, $D_\infty$, $\mathbb Z_2^2$, $\mathbb Z_2 * \mathbb Z_2 * \mathbb Z_2$, $\mathbb Z_2^2 * \mathbb Z_2$, $\mathbb Z_2\oplus D_\infty$ or $\mathbb Z_2^3$. It follows from \cite[Proposition 2.1]{MR4093966} that the groups $\mathbb Z_2 * \mathbb Z_2 * \mathbb Z_2$ and $\mathbb Z_2^2 * \mathbb Z_2$ are not virtually abelian. Further, the remaining groups satisfy the desired graph-theoretic property, which establishes the forward implication when $|V(\Gamma)|\leq 3$.
\par

Now, assume that $V(\Gamma) \geq 4$. If $\Gamma$ is complete, there is nothing to prove. Suppose that $\Gamma$ is not complete. Then there exist two non-adjacent vertices $v_1$ and $v_2$. Let $u$ be any other vertex of $\Gamma$. Suppose that $u$ is not adjacent to at least one of $v_1$ or $v_2$. Then the group $W(\Gamma)$ surjects onto the standard parabolic subgroup generated by $\{v_1,v_2,u\}$ which is isomorphic to either $\mathbb Z_2 * \mathbb Z_2 * \mathbb Z_2$ or $\mathbb Z_2^2 * \mathbb Z_2$. Since a quotient of a virtually abelian group is virtually abelian, and the groups $\mathbb Z_2 * \mathbb Z_2 * \mathbb Z_2$ and $\mathbb Z_2^2 * \mathbb Z_2$ are not virtually abelian, we get a contradiction. Hence, $v_1$ and $v_2$ are adjacent to every other vertex of $\Gamma$. This shows that $\Gamma$ is the join of the edgeless graph on $v_1, v_2$ and a subgraph $\Gamma^{'}$ spanned by the remaining vertices. The group $W(\Gamma^{'})$ is virtually abelian being a quotient of $W(\Gamma)$, and hence we can repeat the preceding argument on $W(\Gamma^{\prime})$ leading to the desired condition on the graph $\Gamma$. Conversely, if $\Gamma$ is the join of a complete graph on $m$ vertices for some $m\geq 0$ and $n$ copies of a graph with two isolated vertices, then $W(\Gamma)$ is isomorphic to $\mathbb{Z}_2^m \oplus  D_{\infty}^n$, which clearly  has a finite index abelian subgroup of rank $n$. 
\end{proof}

\begin{proposition}\label{csp fails for T_n}
An infinite irreducible right-angled Coxeter group of rank more than two does not have the congruence subgroup property.  In particular, the congruence subgroup property fails for $T_n$ when $n\geq 4$.
\end{proposition}

\begin{proof}
Let $W$ be an infinite irreducible right-angled Coxeter group of rank $k \ge 3$ and $\Gamma$ the graph such that $W=W(\Gamma)$. Suppose that $W$ has a finite index abelian subgroup of rank $n\geq 0$.  Then, by Lemma \ref{virtual abelian coxeter group criteria}, the graph $\Gamma$ is the join of a complete graph on $m$ vertices for some $m\geq 0$ and $n$ copies of a graph with two isolated vertices. If $k\neq 2n$, then $m\geq 1$. Thus, there exists a vertex that is adjacent to every other vertex of $\Gamma$. Consequently, $W$ decomposes into a direct product of $\mathbb Z_2$ with another Coxeter group, a contradiction to the irreducibility of $W$. If $k=2n$, then $W(\Gamma)$ is isomorphic to $\bigoplus_n D_\infty$ where $n\geq 2$ (since we assumed $k\geq 3$) which is not possible as $W$ is irreducible. Hence, $W$ is not virtually abelian, and the assertion follows from Theorem \ref{virtual abelian and csp}.
\end{proof}

\begin{proposition}\label{csp fails for L_n}
A virtually free small Coxeter group does not have the congruence subgroup property.  In particular, the congruence subgroup property fails for $L_n$ when $n\geq 4$.
\end{proposition}
\begin{proof}
Let $W$ be a virtually free small Coxeter group and let $H$ be its finite index free subgroup of rank $n \ge 2$. Then $K:= H\cap E(W)$ is a finite index free subgroup (being a subgroup of $H$) of $E(W)$, and hence it is non-trivial. Furthermore, $K$ is non-abelian free since it is either $H$ or a free subgroup of $H$ of rank $1+2(n-1)$ (by the Nielsen-Schreier Theorem). By Proposition \ref{necesarycsp}, there exists a finite index subgroup $P$ of $\rho(E(W))$, which does not contain any principal congruence subgroup of $\rho(E(W))$. It follows that $\rho^{-1}(P)$ does not contain any principal congruence subgroup $\rho^{-1}(\rho(W) \cap \GL(n,\mathbb Z)[m])$ of $W$. Since $\rho^{-1}(P)$ is also a finite index subgroup of $W$,  the first assertion follows. 
\par
For $n \ge 4$, consider the space
$$
Y_n=\mathbb{R}^n \setminus \{(x_1, x_2,\ldots,x_n)\in \mathbb{R}^n ~\mid~ x_i=x_j,~x_k=x_l,~\textrm{where}~ i,j,k,l ~\text{are pairwise distinct} \}.
$$
By \cite[Proposition 3.2(b)]{MR1386845}, $Y_n$ is an Eilenberg--MacLane space for the pure triplet group $PL_n$. Further, by \cite[Theorem 2.1]{MR1386845}, the space $Y_n$ is homotopy equivalent to a bouquet of circles, and hence $PL_n$ is a free group of finite rank for $n \geq 4$. Since  $PL_n$ is of finite index in $L_n$, the second assertion follows.
\end{proof}

The following problem seems natural.

\begin{problem}
Determine finite index subgroups of  $T_n$ and $L_n$ ($n \ge 4$) which contain no principal congruence subgroups.
\end{problem}
\medskip

\subsection{Congruence subgroups of twin groups}
For $n \ge 3$ and $m \ge 2$, let $T_n[m]:=\ker(\rho_m)$ denote the principal congruence subgroup of $T_n$ of  level $m$. We shall need the following number theoretic observation.

\begin{lemma}\label{pm}
Let $m\geq 3$,  $f_m(Y) = Y^m-1$ and $g(Y) = (Y - 1)^3$. Then $$f_m(Y) = p_m(Y) \mod g(Y),$$ where $$p_m(Y)=\frac{m(m-1)}{2}Y^2-m(m-2)Y+\frac{m(m-3)}{2}.$$
\end{lemma}

\begin{proof}
The base case $m=3$ is a direct check. Assuming that  $f_m(Y) = p_m(Y) \mod g(Y)$ for $m>3$, we see that
\begin{eqnarray*}
 f_{m+1}(Y) &=& Y^{m+1} - 1\\
&=& \left[Y\left( \frac{m(m-1)}{2}Y^2-m(m-2)Y+\frac{m(m-3)}{2} +1\right) - 1\right] \mod g(Y)\\
&=& \left[ \frac{m(m-1)}{2}Y^3-m(m-2)Y^2+\frac{m(m-3)}{2}Y + Y - 1\right] \mod g(Y)\\
&=& \left[ \frac{m(m-1)}{2}(3Y^2-3Y+1) - m(m-2)Y^2+\frac{m(m-3)}{2}Y + Y - 1\right] \mod g(Y),\\
&& \textrm{using the base case}~ m=3\\
&=& \left[ \frac{m(m+1)}{2}Y^2+ (1-m)(1+m)Y +  \frac{(m+1)(m-2)}{2}\right] \mod g(Y)\\
&=& p_{m+1}(Y) \mod g(Y).
\end{eqnarray*}
This concludes the proof.
\end{proof}

For $n, m \ge 2$, let $W_{n,m}$ be the Coxeter group given by the presentation
$$W_{n,m}=\left\langle u_1,\ldots, u_{n-1} ~\mid~ u_i^2=1, \quad (u_iu_{i+1})^m=1, \quad (u_iu_j)^2=1 \text{ for } |i-j|\geq 2 \right\rangle.$$

\begin{proposition}\label{L1}
The following assertions hold:
\begin{enumerate}
    \item If $m\geq 2$, then the map $\kappa:W_{n,m}\to \GL(n-1, \mathbb Z_{2m})$ given by $\kappa(u_i)=X_i \mod 2m$ is a group homomorphism.
    \item If $m\geq 3$, then the map $\phi:W_{n,m}\to \GL(n-1, \mathbb Z_m)$ given by $\phi(u_i)=X_i \mod m$ is a group homomorphism.
\end{enumerate}
\end{proposition}
\begin{proof}
To show that $\kappa$ is a homomorphism, it suffices to show that $(X_iX_{i+1})^m=I$ in $\GL(n-1,\mathbb Z_{2m})$ for each $i$. 
 Note that
\begin{equation*}
X_iX_{i+1} =
\begin{cases}
\left(\begin{array}{ccc|c} 3 & -2 & 4 & 0 \\ 2 & -1 & 2 & 0\\ 0 & 0 & 1 & 0\\ \hline 0 & 0 & 0 & I_{n-4}\end{array}\right) & \text{if}\; i= 1,\\
\\

\left(\begin{array}{c|cccc|c}I_{i-2} & 0 & 0 & 0 & 0 & 0 \\ \hline & 1 & 0 & 0 & 0 & \\ 0 & 2 & 3 & -2 & 4 & 0 \\ & 0 & 2 & -1 & 2 & \\ 0 & 0 & 0 & 0 & 1 & 0\\ \hline 0 & 0 & 0 & 0 & 0 & I_{n-(i+3)}\end{array}\right) & \text{if}\;  2\leq i \leq n-3,\\
\\

\left(\begin{array}{c|ccc}I_{n-4} & 0 & 0 & 0 \\ \hline 0 & 1 & 0 & 0 \\ 0 & 2 & 3 & -2 \\0 & 0 & 2 & -1  \end{array}\right) & \text{if}\; i= n-2,
\end{cases}
\end{equation*}
and 
\begin{equation*}
X_iX_{i+1} - I =
\begin{cases}
\left(\begin{array}{ccc|c} 2 & -2 & 4 & 0 \\ 2 & -2 & 2 & 0\\ 0 & 0 & 0 & 0\\ \hline 0 & 0 & 0 & 0\end{array}\right) & \text{if}\; i= 1,\\
\\

\left(\begin{array}{c|cccc|c}0 & 0 & 0 & 0 & 0 & 0 \\ \hline & 0 & 0 & 0 & 0 & \\ 0 & 2 & 2 & -2 & 4 & 0 \\ & 0 & 2 & -2 & 2 & \\ 0 & 0 & 0 & 0 & 0 & 0\\ \hline 0 & 0 & 0 & 0 & 0 & 0\end{array}\right) & \text{if}\;  2\leq i \leq n-3,\\
\\

\left(\begin{array}{c|ccc}0 & 0 & 0 & 0 \\ \hline 0 & 0 & 0 & 0 \\ 0 & 2 & 2 & -2 \\0 & 0 & 2 & -2  \end{array}\right) & \text{if}\; i= n-2.
\end{cases}
\end{equation*}
Set $Y = X_i X_{i+1}$. It is easy to see that $(Y-I)^3=0$ since  it is the minimal polynomial of $Y$. We claim that $Y^m= I \mod 2m$. Write $m=2^jk$, where $j\geq 0$ and $k\geq 1$ is an odd integer. We use induction on $j$ to prove our claim. First, consider the base case $j=0$. In this case, $m = k$ is odd. By Lemma \ref{pm}, we have $Y^m-I= p_m(Y)$. Since $m$ is odd, it follows that $p_m(Y)=0 \mod m$, and hence $Y^m-I=0 \mod m$. Note that $Y=I \mod 2, $ and therefore $Y^m=I \mod 2$. Since $m$ is odd, we obtain  $Y^m = I \mod 2m$. 
\par

Now assume that the claim holds for $j-1$ for $j>0$. By induction hypothesis, we have $Y^{2^{j-1}k}-I=0 \mod 2 (2^{j-1}k)= m$. Furthermore, we have $(Y^{2^{j-1}k}+I)=0 \mod 2$.  Consequently, it follows that $$Y^m-I=(Y^{2^{j-1}k}+I) (Y^{2^{j-1}k}-I) =0 \mod 2m.$$ This shows that $\kappa$ is well-defined, and hence a homomorphism. Note that $\phi$ is simply the composition of $\kappa$ with the modulo $m$ reduction homomorphism $\GL(n-1, \mathbb Z_{2m})\to \GL(n-1, \mathbb Z_{m})$, and hence is a homomorphism.
 \end{proof}

\begin{corollary}\label{Cor1}
$\llangle \left(s_is_{i+1}\right)^m ~\mid~ 1 \le i \le n-2\rrangle~ \unlhd ~T_n[2m]$ for each $m \ge 2$.
\end{corollary}

\begin{proof}
Let $\pi_m: T_n\to W_{n,m}$ denote the natural projection $s_i \mapsto u_i$. Since the diagram
\[
\begin{tikzcd}
 {T_n} \arrow{d}{\pi_m} \arrow{rd}{\rho_{2m}} &\\ 
W_{n,m} \arrow{r}{\kappa} & \GL(n-1, \mathbb Z_{2m}).
\end{tikzcd}
\]
commutes, the assertion follows immediately.
\end{proof}

\begin{proposition}
The following  assertions hold:
\begin{enumerate}
    \item $\kappa:W_{n,m}\to \GL(n-1, \mathbb Z_{2m})$ is injective if and only if $m=2,3$.
    \item $\phi:W_{n,m}\to \GL(n-1, \mathbb Z_m)$ is injective if and only if $m=3$.
\end{enumerate}
\end{proposition}

\begin{proof}
Observe that $W_{n,2} \cong \mathbb Z_2^{n-1}$.  Taking $W=T_n$, the proof of Proposition \ref{racg commutator} shows that the map  $\kappa:W_{n,2}\to \GL(n-1, \mathbb Z_{4})$ is the same as the map $\psi:   \mathbb Z_2^{n-1} \to  \GL(n-1, \mathbb Z_{4})$, and hence $\kappa$ is injective. 
\par
Note that $W_{n,3}\cong S_n$. It is well-known that $A_n$ is the only non-trivial proper normal subgroup of $S_n$ for $n \ne 4$, while $A_4$ and $\{e, (12)(34), (13)(24), (14)(23)\}$ are non-trivial proper normal subgroups of $S_4$. Since the permutation $(12)(n-1, n)$ lies in each non-trivial normal subgroup of $S_n$ and $\kappa((12)(n-1, n)) \neq I$, it follows that $\kappa:W_{n,3}\to \GL(n-1, \mathbb Z_{6})$ is injective. The same reasoning shows that $\phi$ is injective for $m=3$.
\par
Finally, the classification of finite irreducible Coxeter groups implies that $W_{n,m}$ is an infinite group  for $m\geq 4$, and hence neither $\kappa$ nor $\phi$ can be injective in this case.
\end{proof}

\begin{proposition}\label{tn3=ptn}
$T_n[3]=T_n[6] =PT_n$.
\end{proposition}
\begin{proof}
Since $\phi:W_{n,3}\cong S_n\to \GL(n-1, \mathbb Z_3)$ is injective and the diagram
\[
\begin{tikzcd}
 {T_n} \arrow{d}{\pi} \arrow{rd}{\rho_3} &\\ 
 S_n \arrow{r}{\phi} & \GL(n-1, \mathbb Z_3)
\end{tikzcd}
\]
commutes, it follows that $\ker(\rho_3)= \ker(\pi)$. Similarly, since the diagram
\[
\begin{tikzcd}
 {T_n} \arrow{d}{\pi_3} \arrow{rd}{\rho_6} &\\ 
W_{n,3} \arrow{r}{\kappa} & \GL(n-1, \mathbb Z_6)
\end{tikzcd}
\]
commutes with $\kappa$ being injective, it follows that $T_n[6]=\ker (\rho_6)= \ker (\pi_3)= PT_n$.
\end{proof}

Corollary \ref{prin cong torsion free} and Proposition \ref{tn3=ptn} recover the following result of \cite[Theorem 3]{MR4027588}.

\begin{corollary}
$PT_n$ is torsion free for $n \ge 3$.
\end{corollary}

\begin{remark}
Let $W= \langle S \mid R \rangle$, where $|S|=n-1$ and $n \ge 3$. Suppose that $\rho_k: W \to \GL(n-1,\mathbb Z_k)$ is not surjective. If $k \mid m$, then $\rho_k$ can be seen as the composition $$W\stackrel{\rho_m}{\longrightarrow}  \GL(n-1,\mathbb Z_m) \longrightarrow \GL(n-1,\mathbb Z_k).$$ Since $\rho_k$ is not surjective, so is $\rho_m$. Specialising to twin groups, note that $T_n[2]=T_n$. Thus, $\rho_2:T_n\to \GL(n-1,\mathbb Z_2)$ is not surjective, and hence $\rho_m$ is not surjective when $m$ is even. By Proposition \ref{tn3=ptn},  $\ker(\rho_3)=PT_n$, and hence $\im(\rho_3)\cong S_n$. Since $|S_n|<|\GL(n-1,\mathbb Z_3)|$, it follows that $\rho_3:T_n\to \GL(n-1,\mathbb Z_3)$ is not surjective, and hence $\rho_m$ is not surjective when $m$ is a multiple of $3$. 
\end{remark}

\begin{proposition}
$T_3[m] \cong \mathbb Z$ for each $m\geq 3$. More precisely,
$$
\begin{aligned}
T_3[m]  &=\left\{\begin{array}{ll}
 \langle (s_1s_2)^m \rangle & ~\text {if} ~m~\textrm{is odd},\\
\langle (s_1s_2)^{m/2}\rangle  & ~\text {if} ~m~\textrm{is even}.
\end{array}\right.
\end{aligned}
$$
\end{proposition}

\begin{proof}
Note that $T_3 \cong D_\infty$, the infinite dihedral group. It is easy to see that the only non-trivial normal subgroups of $T_3$ are $T_3$,  $\llangle s_1  \rrangle$, $\llangle s_2 \rrangle$ and $\langle (s_1s_2)^k \rangle$ for $k\geq 1$. It follows from Lemma \ref{length} that $T_3[m] = \langle (s_1s_2)^k \rangle \cong \mathbb{Z}$ for some $k\geq 1$. Recall that  
$$\rho(s_1) = \begin{bmatrix}
-1 & 2 \\
0 & 1
\end{bmatrix},
\quad
 \rho(s_2) =
\begin{bmatrix}
1 & 0 \\
2 & -1 
\end{bmatrix}$$
and hence
 $$\rho(s_1s_2) = \begin{bmatrix}
3 & -2 \\
2 & -1 
\end{bmatrix}.$$
Induction on $k$ shows that $$\rho((s_1s_2)^k) \; = \;\;
\begin{bmatrix}
2k+1 & -2k \\
2k & 1-2k
\end{bmatrix}$$ for all $k\geq 1$.  Noting that $\rho((s_1s_2)^k)$ is the identity modulo $m$ if and only if $m$ divides $2k$, the last assertion is now immediate.
\end{proof}

\begin{corollary}
$T_3$ has the congruence subgroup property.
\end{corollary}

\begin{proof}
It is easy to see that any finite index subgroup of $T_3$ is of the form $\langle(s_1s_2)^k\rangle$ or $\langle (s_1s_2)^k, w \rangle$ for some $k\in \mathbb Z$ and some word $w \in T_3$ of odd length. In each case, it contains a principal congruence subgroup of $T_3$.
\end{proof}

Let $E(T_n)$ be the index two subgroup of $T_n$ consisting of all words of even lengths.

\begin{proposition}\label{propgcd}
Let $m, k \ge 3$ be co-prime integers. Then the following  assertions hold:
\begin{enumerate}
\item $T_n[mk]=T_n[m]\cap T_n[k]$.
\item  $T_n[m]~ T_n[k]=E(T_n)$. 
\item $T_n[m]/T_n[mk]\cong E(T_n)/T_n[k]$.
\item $E(T_n)/T_n[mk]\cong E(T_n)/T_n[m]\times E(T_n)/T_n[k]$.
\end{enumerate}
\end{proposition}

\begin{proof}
Let $X$ be a matrix. Then $X=I \mod m$ and $X=I \mod k$ if and only if $X =I \mod mk$. Thus, we have $T_n[mk]=T_n[m]\cap T_n[k]$, which is assertion (1).
\par
By Lemma \ref{length}, we already have $T_n[m]~ T_n[k]\le E(T_n)$. Note that $E(T_n)$ is generated by $\{s_1 s_2, s_2s_3, \ldots, s_{n-2}s_{n-1}\}$. Since $m$ and $k$ are co-prime, there exist integers $a$ and $b$ such that $am+bk=1$. Using Corollary \ref{Cor1}, we see that
$$s_is_{i+1}= (s_is_{i+1})^{am+bk}  =\left((s_is_{i+1})^m\right)^a\left((s_is_{i+1})^k\right)^b \in T_n[m]~ T_n[k],$$ which proves assertion (2). Assertion (3) follows from (1) and (2). If $M$ and $N$ are normal subgroups of a group $G$ such that $G=MN$, then $G/(N\cap M) \cong (G/M) \times (G/N)$. Using this fact together with  (1) and (2) gives $$E(T_n)/T_n[mk]= E(T_n)/\left(T_n[m]\cap T_n[k]\right)\cong (E(T_n)/T_n[m]) \times (E(T_n)/T_n[k]),$$ which is assertion (4).
\end{proof}

\begin{theorem}\label{quotient twin cong subgroup}
The following  assertions hold:
\begin{enumerate}
\item If $3 \nmid m$ and $A_n$ is the alternating group, then $$T_n[m]/T_n[3m]\cong A_n.$$
\item If $m$ is odd and $H$ is the subgroup of $W_{n, 2}$ consisting of words of even lengths, then $$T_n[m]/T_n[4m]\cong H.$$
\item If $m$ is odd and $3 \nmid m$, then $$T_n[m]/T_n[12m]\cong A_n\times H.$$
\end{enumerate}
\end{theorem}
\begin{proof}
Let  $\pi:T_n\to S_n$ be the natural map. It follows from Lemma \ref{length} that $\pi(T_n[m]) \le A_n$, the alternating group. By Corollary \ref{Cor1}, we have $(s_is_{i+1})^m \in T_n[m]$, and hence $\pi((s_is_{i+1})^m)\in \pi(T_n[m])$.
But, note that
$$
\begin{aligned}
\pi((s_is_{i+1})^m) &=\left\{\begin{array}{ll}
(i, i+1, i+2) & \text { if } m=3k+1, \\
(i, i+2, i+1) & \text { if } m=3k+2,\\
1  & \text { if } m=3k.
\end{array}\right.
\end{aligned}
$$
Since $\{(i, i+1, i+2) \mid 1 \le i \le n-2\}$ is a generating set for $A_n$ and $\gcd(3,m)=1$, it follows that $\pi(T_n[m])=A_n$. Considering the restriction $\pi|_{T_n[m]}$, we see that
$$\ker(\pi|_{T_n[m]})=\ker(\pi) \cap T_n[m]= PT_n \cap T_n[m]= T_n[3]\cap T_n[m]=T_n[3m],$$
which gives the desired result.
\par

For assertion (2), consider the abelianisation map $\pi_2:T_n\to W_{n, 2}\cong \mathbb Z_2^{n-1}$. By Corollary \ref{Cor1},  $(s_is_{i+1})^m \in T_n[m]$, and hence $\pi_2((s_is_{i+1})^m)\in \pi_2(T_n[m])$. Since $m$ is odd, we have 
$$\pi_2((s_is_{i+1})^m) =\pi_2(s_is_{i+1})=u_iu_{i+1}.$$ 
Note that $H$ is generated by $\{u_1u_2, u_2u_3, \ldots, u_{n-2}u_{n-1} \}$, and hence $H \le \pi_2(T_n[m])$. It now follows from Lemma \ref{length} that $H = \pi_2(T_n[m])$. Considering the restriction $\pi_2|_{T_n[m]}$, we see that
$$\ker(\pi_2|_{T_n[m]})=\ker(\pi_2) \cap T_n[m]= T_n[4]\cap T_n[m]=T_n[4m],$$
which proves the assertion. 
\par

Finally, for assertion (3), we have
\begin{eqnarray*}
T_n[m]/T_n[12m] &\cong & \left(T_n[m]T_n[12]\right)/T_n[12]\\
&\cong & E(T_n)/T_n[12] \\
&\cong & (E(T_n)/T_n[3]) \times (E(T_n)/T_n[4]) \\
&\cong& A_n\times H,
\end{eqnarray*}
where the isomorphisms follow from Proposition \ref{propgcd}.
\end{proof}

\subsection{Congruence subgroups of triplet groups}
For $n \ge3$ and $m \ge 2$, let $L_n[m]:=\ker(\rho_m)$ denote  the principal congruence subgroup of $L_n$ of level $m$. Let $\pi:L_n\to  S_n$ be the natural map and let $PL_n=\ker(\pi)$.

\begin{proposition}\label{ln2=pln}
$L_n[2]=PL_n$ for $n \ge 3$.
\end{proposition}

\begin{proof}
For $n \ge 3$, we claim that the map $\phi:S_n\to \GL(n-1, \mathbb Z_2)$ given by $\phi(\tau_i)=Y_i \mod 2$ is an embedding. Note that 
\begin{equation*}
\phi(\tau_i) = 
\begin{cases}
\left(\begin{array}{cc|c}
1 & 1 & 0 \\
0 & 1 & \\
\hline 0 & & I_{n-3}
\end{array}\right) & \text{if}\; i= 1,\\
\\
\left(\begin{array}{c|ccc|c}I_{i-2} &  & 0 &  & 0 \\ \hline & 1 & 0 & 0 &\\ 0 & 1 & 1 & 1 & 0 \\ & 0 & 0 & 1 & \\ \hline 0 &  & 0 &  & I_{n-(i+2)}\end{array}\right) & \text{if}\;  2\leq i \leq n-2,\\
\\
\left(\begin{array}{c|c}
I_{n-3} & 0 \\
\hline
0 & \begin{array}{cc}
1 & 0 \\
1 & 1
\end{array}
\end{array}\right) & \text{if}\; i= n-1.
\end{cases}
\end{equation*}
Taking
$U=\left(\begin{matrix}
1 & 0 & 0 \\
1 & 1 & 1 \\
0 & 0 & 1
\end{matrix}\right),$
we see that $U^2= I_3 \mod 2$. Writing $\phi(\tau_i)$ and $\phi(\tau_j)$ as block matrices made up of $U$ together with identity and zero matrices,  we see that $\phi(\tau_i)\phi(\tau_j)= \phi(\tau_j)\phi(\tau_i)$ for all $|i-j|\geq 2$. Hence, $\phi$ is a homomorphism. Note that, 
$$
\begin{aligned}
\phi(\tau_1\tau_{n-1}) &= 
\begin{cases}
\left(\begin{array}{cc}
2 & 1 \\
1 & 1 
\end{array}\right) & \text{if}\; n= 3,\\
\\
\left(\begin{matrix}
1 & 1 & 0 \\
0 & 1 & 0 \\
0 & 1 & 1
\end{matrix}\right) & \text{if}\; n=4,\\
\\
\left(\begin{array}{cc|c|cc}
1 & 1 & 0 & 0 & 0 \\
0 & 1 & 0 & 0 & 0 \\\hline
0 & 0 & I_{n-5} & 0 & 0 \\\hline
0 & 0 & 0 & 1 & 0 \\
0 & 0 & 0 & 1 & 1
\end{array}\right) & \text{if}\; n\geq 5,\\
\end{cases}\\
\end{aligned}
$$
and hence $\phi(\tau_1\tau_{n-1})\neq I_{n-1}$. Since the only proper normal subgroup of $S_n$ that does not contain $\tau_1\tau_{n-1}$ is the trivial subgroup, it follows that $\phi$ is injective. Finally, the commutativity of the diagram
\[
\begin{tikzcd}
 {L_n} \arrow{d}{\pi} \arrow{rd}{\rho_2} &\\ 
 S_n \arrow{r}{\phi} & \GL(n-1, \mathbb Z_2)
\end{tikzcd}
\]
implies that $L_n[2]=\ker(\rho_2)= \ker(\pi)=PL_n$.
\end{proof}

The following problem seems challenging.

\begin{problem}
Find generating sets for principal congruence subgroups of $T_n$ and $L_n$.
\end{problem}

\medskip

\section{Crystallographic quotients of small Coxeter groups}\label{sec4}
A closed subgroup $H$ of a Hausdorff topological group $G$ is said to be {\it uniform} if $G/H$ is a compact space.  A discrete and uniform subgroup $G$ of $\mathbb{R}^n \rtimes \Oo(n, \mathbb{R})$ is called a {\it crystallographic group of dimension $n$}. If, in addition $G$ is torsion free, then it is called a {\it Bieberbach group of dimension $n$}. The following characterisation of crystallographic groups is well-known \cite[Lemma 8]{MR3595797}.

\begin{lemma}
 A group $G$ is a crystallographic group if and only if there is an integer $n$, a finite group $H$, and a short exact sequence
$$
0 \longrightarrow \mathbb{Z}^n \longrightarrow G \stackrel{\eta}{\longrightarrow} H \longrightarrow 1
$$
such that the integral representation $\Theta : H \longrightarrow \Aut\left(\mathbb{Z}^n\right)$ defined by $\Theta(h)(x)=z x z^{-1}$ is faithful, where  $h \in H$, $x \in \mathbb{Z}^n$ and $z \in G$ is such that $\eta(z)=h$. 
\end{lemma}

The group $H$ is called the holonomy group of $G$, the integer $n$ is called the dimension of $G$, and $\Theta : H \longrightarrow \Aut\left(\mathbb{Z}^n\right)$ is called the holonomy representation of $G$. It is known that any
finite group is the holonomy group of some flat manifold \cite[Theorem III.5.2]{MR0862114}. Furthermore, there is a correspondence between the class of Bieberbach groups and the class of compact flat Riemannian manifolds \cite[Theorem 2.1.1]{MR1482520}. Crystallographic quotients of braid groups and their virtual analogues have been investigated in recent works.

\begin{proposition}\cite[Proposition 1]{MR3595797}
 For $n \geq 2$, there is a short exact sequence
$$
1 \longrightarrow \mathbb{Z}^{n(n-1) / 2} \longrightarrow B_n /P_n^{'} \longrightarrow S_n \longrightarrow 1
$$
such that the group $B_n /P_n'$ is a crystallographic group of dimension $n(n-1)/2$.
\end{proposition}

\begin{proposition}\cite[Theorem 2.4]{Ocampo-Santos-2021}
 For $n \geq 2$, there is a split short exact sequence
$$
1 \longrightarrow \mathbb{Z}^{n(n-1)} \longrightarrow V B_n / V P_n^{'} \longrightarrow S_n \longrightarrow 1
$$
such that the group $V B_n / V P_n^{'}$ is a crystallographic group of dimension $n(n-1)$.
\end{proposition}
\medskip

\subsection{Crystallographic quotients of twin groups}

Though determining the precise structure of $PT_n$ is still an open problem (see the introduction), the following result regarding its abelianisation is well-known (see \cite[Theorem 9.9]{MR2193190} and  \cite{MR1317619}).

\begin{proposition}
The abelianisation of the pure twin group $PT_n$ is a free abelian group of rank $2^{n-3}(n^2-5n+8)-1$.
\end{proposition}

For each $n\geq 2$, we have a well-defined homomorphism $f_n: PT_n \to PT_{n-1}$ obtained by removing the $n$-th strand from the diagram of an element of $PT_n$.  Let $U_n$ denote $\ker(f_n)$. In the reverse direction, let $i_n: PT_{n-1}\to PT_n$ be the inclusion obtained by adding a strand to the rightmost side of the diagram of an element of $PT_{n-1}$. Since $f_n ~ i_n= \id_{PT_{n-1}}$, we have the split exact sequence
$$
1\longrightarrow U_n \longrightarrow PT_n \stackrel{f_n}{\longrightarrow}  PT_{n-1} \longrightarrow 1.
$$

Since $PT_n^{'}$ is normal in $T_n$, we have the following short exact sequence
$$
1 \longrightarrow PT_n/PT_n^{'} \longrightarrow T_n /PT_n^{'} \stackrel{\pi}{\longrightarrow} S_n \longrightarrow 1.
$$

For subgroups $A$ and $B$ of a group $G$,  let $[A, B] = \langle [a, b] \mid a \in A,~ b \in B\rangle$. The following basic observation is needed in the main result of this section.

\begin{lemma}\label{TL1}
Let $G=N\rtimes H$ be the internal semi-direct product of a normal subgroup $N$ and a subgroup $H$ of $G$. Then 
$$G' = N^{'}[N,H] H^{'} = \big(N^{'} [N,H] \big) \rtimes H^{'}.$$ In particular, $G^{'} \cap H=H^{'}.$
\end{lemma}

\begin{theorem}\label{tn mod ptn cryst}
For $n\geq 4$, there is a short exact sequence 
$$1 \longrightarrow PT_n/PT_n^{'} \longrightarrow T_n /PT_n^{'} \stackrel{\pi}{\longrightarrow} S_n \longrightarrow 1$$ such that $T_n /PT_n^{'}$ is a crystallographic group of dimension $2^{n-3}(n^2-5n+8)-1$. Furthermore, $T_n /PT_n^{'}$ is not a  Bieberbach group. 
\end{theorem}

\begin{proof}
Note that $PT_3= \langle (s_1s_2)^3 \rangle$ is the infinite cyclic group \cite[Theorem 2]{MR4027588}, and hence $T_3$ is not a crystallographic group.  For $n \ge 4$, it suffices to show that $\Theta(\tau_1\tau_{n-1})$ is a non-identity automorphism of $PT_n/PT_n^{'}$. For $n=4$, we have $\Theta(\tau_1\tau_3)((s_1s_2)^3)=s_3\left(s_2s_1\right)^3s_3 \mod PT_4^{'}$. Since $(s_1s_2)^3$ and $s_3 \left(s_1s_2\right)^3 s_3$ lie in a minimal generating set for $PT_4 \cong F_7$ \cite[Theorem 2]{MR4027588}, it follows that $(s_1s_2)^3\neq s_3\left(s_2s_1\right)^3s_3 \mod PT_4^{'}$. But, the only proper normal subgroup of $S_4$ that does not contain $\tau_1\tau_3$ is the trivial subgroup, and hence $\Theta$ is faithful.
\par

Now, assuming $n\geq 5$, we have $\Theta(\tau_1\tau_{n-1})((s_1s_2)^3)=(s_2s_1)^3 \mod PT_n^{'}$. We claim that $(s_1s_2)^3\neq (s_2s_1)^3 \mod PT_n^{'}$, that is, $(s_1s_2)^6\not\in PT_n^{'}$. We proceed by induction on $n$. Note that $PT_5$ is the free group of rank 31 with $(s_1s_2)^3$ lying in its free generating set \cite[Section 5]{MR4170471}. Since squares of free generators do not lie in the commutator subgroup, it follows that $(s_1s_2)^6\not\in PT_5^{'}$, which proves the base case of the induction. Now, assume that $(s_1s_2)^6\not\in PT_{n-1}^{'}$ for some $n>5$. Since $PT_n=U_n \rtimes PT_{n-1}$ and $(s_1s_2)^6\in PT_{n-1}$, it follows from Lemma \ref{TL1} that $(s_1s_2)^6 \not\in PT_n^{'}$, which completes the proof.
\par

Finally, for $n \ge 4$, the order two element $s_1s_3\in T_n$ gives an order two element in  $T_n/PT_n^{'}$, and therefore $T_n /PT_n^{'}$ is not a  Bieberbach group. 
\end{proof}
\medskip

Next, we consider another crystallographic quotient of $T_n$ arising from its abelianisation exact sequence. For $1 \le i \le n -2$ and  $1\leq p < j  \le n -2$, we set
$$ \beta_0(i)=s_{i+1}s_{i}s_{i+1}s_{i}$$
and 
$$\beta_p(j)=  s_{j-p}\cdots s_{j-1}(s_{j+1}s_{j}s_{j+1}s_{j})s_{j-1}\cdots s_{j-p}.$$
A delicate use of the Reidemeister-Schreier method gives the following result \cite[Theorem 1.1]{MR3943376}.

\begin{theorem}
For each $n \geq 3$, $T_n^{'}$ has a finite presentation with generating set 
$$\{ \beta_p(j) ~\mid~ 0 \leq p<j \leq n-2 \}$$ 
and defining relations
$$\beta_{j-k}(j) \beta_{t-(j+l)}(t)=\beta_{t-(j+l)}(t) \beta_{j-k}(j) \quad \textrm{and} \quad \beta_{t-k}(t)=\beta_{j-k}(j)^{-1} \beta_{t-(j+1)}(t) \beta_{j-k}(j)$$
for $3 \le l$, $1 \leq k \leq j$ and $j+2 \leq t \leq n-2$.
\end{theorem}

For $m, t \in\{1,2,  \ldots, n-2\}$ with $1 \leq m \leq t-2$, we can express $\beta_{t-m}(t)$ as the conjugate of $\beta_{t-(m+1)}(t)$ by $\beta_0(m)$. Iterating the process, we can express $\beta_{t-m}(t)$ as the conjugate of $\beta_1(t)$ by the element $\beta_0(t-2)\cdots \beta_0(m)$. Thus, one obtains the following proposition \cite[Lemma 5.3 and Corollary 1.3]{MR3943376}.

\begin{proposition}
For each $n \geq 3$, $T_n^{'}$ has a finite presentation with $(2n-5)$ generators $$\{ \beta_0(1), ~\beta_0(j), ~\beta_1(j) ~\mid~ 2 \leq j \leq n-2\},$$ and  $T_n^{'}/T_n^{''}  \cong \mathbb{Z}^{2n-5}.$
\end{proposition}

We set a presentation for the abelianisation 
$$T_n/T_n^{'}\cong \mathbb{Z}_2^{n-1}=\langle a_1,a_2,\ldots,a_{n-1} ~\mid~ [a_i,a_j]=1 \text{ for all } i,j \rangle.$$ 
We can identify the precise holonomy representation of $T_n/T_n^{''}$.  Modulo $T_n^{''}$, the action of the generators of $\mathbb Z_2^{n-1}$ on the generators $\beta_p(j)$ $(p=0,1)$  of $T_n^{'}$ is given by
\begin{equation}
\begin{aligned}
\Theta(a_k)(\beta_0(j))&=\begin{cases} \beta_0(j) &\text { if } j\neq k-2, k-1, k, k+1, \\
\beta_0(j+1)\beta_0(j)\beta_1(j+1)^{-1} &\text { if } j = k-2, \\
\beta_0(j)^{-1} &\text { if } j = k-1, k, \\
\beta_1(j) &\text { if } j = k+1, 
\end{cases} \\
\Theta(a_k)(\beta_1(j))&= \begin{cases} \beta_1(j) &\text { if } j\neq k-2, k-1,k,k+1, \\
\beta_0(j+1)\beta_1(j)\beta_1(j+1)^{-1} &\text { if } j = k - 2, \\
\beta_1(j)^{-1} &\text { if } j = k-1, k, \\
\beta_0(j) &\text { if } j = k + 1. 
\end{cases}
\end{aligned}
\end{equation}

\begin{theorem}\label{tn mod tn' cryst}
For $n\geq 3$, there is a short exact sequence 
$$1\to T_n^{'}/T_n^{''} \to T_n/T_n^{''} \to  \mathbb{Z}_2^{n-1} \to 1$$
such that $T_n/T_n^{''}$ is a crystallographic group of dimension $2n-5$. Furthermore, $T_n/T_n^{''}$ is not a Bieberbach group and admits only 2-torsion.
\end{theorem}

\begin{proof}
Since $T_n^{''}$ is normal in $T_n$, the abelianisation sequence  gives the desired short exact sequence
$$1\to T_n^{'}/T_n^{''} \to T_n/T_n^{''} \to  \mathbb{Z}_2^{n-1} \to 1.$$
Each element of $\mathbb{Z}_2^{n-1}$ can be written uniquely in the form $a_{i_{k-1}}a_{i_1}a_{i_2}\cdots a_{i_k}$, where $i_1<i_2<\cdots < i_k$. We claim that the map $\Theta(a_{i_{k-1}}a_{i_1}a_{i_2}\cdots a_{i_k})$ is not the identity. We proceed as per the following cases:
\par

Case 1. If $i_k\leq n-2$, then we have
$$	\begin{aligned}
	     \Theta(a_{i_{k-1}}a_{i_1}a_{i_2}\cdots a_{i_k})(\beta_0(i_{k}))&=s_{i_{k-1}} s_{i_1}\cdots s_{i_k}\left(s_{i_{k}+1}s_{i_{k}}s_{i_{k}+1}s_{i_{k}}\right)s_{i_k}\cdots s_{i_2} s_{i_1}s_{i_{k-1}}\\
	     &=s_{i_{k-1}} s_{i_k} s_{i_{k}+1}s_{i_{k}}s_{i_{k}+1}s_{i_{k-1}}\\
	     &=\begin{cases} \beta_0(i_{k})^{-1} & \text { if } i_{k}-1 \neq i_{k-1}, \\
	      \beta_1(i_{k})^{-1} & \text { if } i_{k}-1 = i_{k-1}. \end{cases}
\end{aligned}	    
$$
\par

Case 2. Suppose that $i_k = n-1$. If $i_{k-2} < n-4$ and $i_{k-1}>n-5$, then we have
\begin{eqnarray*}
& & \Theta(a_{i_{k-1}}a_{i_1}a_{i_2}\cdots a_{i_{k-2}} a_{n-1})\left(\beta_1(n-2)\beta_0(n-3)^{-1}\beta_0(n-2)^{-1}\right) \\
&=& s_{i_{k-1}}\left(s_{n-3}s_{n-2}s_{n-3}s_{n-2}\right)s_{i_{k-1}}\\
&=& \begin{cases}  
	     s_{n-2}s_{n-3}s_{n-2}s_{n-3} & \text { if } i_{k-1}=n-2, \\	     	 
	     s_{n-2}s_{n-3}s_{n-2}s_{n-3} & \text { if } i_{k-1}=n-3, \\
	     s_{n-4}\left(s_{n-3}s_{n-2}s_{n-3}s_{n-2}\right)s_{n-4}	& \text { if } i_{k-1}=n-4, 
	      \end{cases}\\
&=& \begin{cases}  
	      \beta_0(n-3)  & \text { if } i_{k-1}=n-2, \\
	     \beta_0(n-3)	& \text { if } i_{k-1}=n-3, \\
	     \beta_1(n-3)^{-1}	& \text { if } i_{k-1}=n-4. 
	      \end{cases}
\end{eqnarray*}	    
\par
If $i_{k-2} < n-4$  and $i_{k-1}\leq n-5$, then we have
\begin{eqnarray*}
\Theta(a_{i_{k-1}}a_{i_1}a_{i_2}\cdots a_{i_{k-2}} a_{n-1})\left(\beta_0(n-2)\right)
= s_{i_{k-1}}\left(s_{n-2}s_{n-1}s_{n-2}s_{n-1}\right)s_{i_{k-1}}
= \left(\beta_0(n-2)\right)^{-1}.
\end{eqnarray*}
\par

If $i_{k-2} = n-4$  and $i_{k-1}=n-3$, then
$$ \Theta(a_{n-3}a_{i_1}a_{i_2}\cdots a_{n-4} a_{n-1})\left(\beta_0(n-2) \right) = s_{n-3}\left(s_{n-2}s_{n-1}s_{n-2}s_{n-1}\right) s_{n-3}=\beta_1(n-2)^{-1}.$$
\par

If $i_{k-2} = n-4$  and $i_{k-1}=n-2$, then
\begin{eqnarray*}
&& \Theta(a_{n-2}a_{i_1}a_{i_2}\cdots a_{n-4} a_{n-1})\left(\beta_1(n-2)\beta_0(n-3)^{-1}\beta_0(n-2)^{-1}\right)\\ 
&=& \Theta(a_{n-2}a_{i_1}a_{i_2}\cdots a_{n-4} a_{n-1})\left(s_{n-1}\left(s_{n-3}s_{n-2}s_{n-3}s_{n-2}\right) s_{n-1} \right) \\
&= & s_{n-2}s_{i_1}s_{i_2}\cdots s_{n-4} s_{n-1}\left(s_{n-1}\left(s_{n-3}s_{n-2}s_{n-3}s_{n-2}\right) s_{n-1} \right)s_{n-1}s_{n-4}\cdots s_{i_1}s_{n-2}\\
&= & s_{i_1}s_{i_2}\cdots s_{n-4} s_{n-2}s_{n-1}\left(s_{n-1}\left(s_{n-3}s_{n-2}s_{n-3}s_{n-2}\right) s_{n-1} \right)s_{n-1}s_{n-2}s_{n-4}\cdots s_{i_1}\\
&= & s_{i_1}s_{i_2}\cdots s_{n-4} \left(s_{n-2}s_{n-3}s_{n-2}s_{n-3}\right)s_{n-4}\cdots s_{i_1}\\
&= & s_{n-3-p}\cdots s_{n-5}s_{n-4} \left(s_{n-2}s_{n-3}s_{n-2}s_{n-3}\right)s_{n-4}s_{n-5}\cdots s_{n-3-p} \quad \textrm{for some} \quad 1 \le p < n-3\\
&= & \beta_p(n-3)\\
&= & \beta_1(n-3) \mod T_n^{''}.
\end{eqnarray*}	    
\par

If $i_{k-2} = n-3$  and $i_{k-1}=n-2$, then
\begin{eqnarray*}
&& \Theta(a_{n-2}a_{i_1}a_{i_2}\cdots a_{i_{k-3}}a_{n-3} a_{n-1})\left(\beta_1(n-2)\beta_0(n-3)^{-1}\beta_0(n-2)^{-1}\right)\\ 
&=& \Theta(a_{n-2}a_{i_1}a_{i_2}\cdots a_{i_{k-3}}a_{n-3} a_{n-1})\left(s_{n-1}\left(s_{n-3}s_{n-2}s_{n-3}s_{n-2}\right) s_{n-1} \right)\\
&=& s_{n-2}s_{i_1}s_{i_2}\cdots s_{i_{k-3}}s_{n-3} s_{n-1}\left(s_{n-1}\left(s_{n-3}s_{n-2}s_{n-3}s_{n-2}\right) s_{n-1} \right)s_{n-1}s_{n-3}s_{i_{k-3}}\cdots s_{i_1}s_{n-2}\\
&=& s_{n-2}s_{i_1}s_{i_2}\cdots s_{i_{k-3}} \left(s_{n-2}s_{n-3}s_{n-2}s_{n-3}\right)s_{i_{k-3}}\cdots s_{i_1}s_{n-2}\\
&=& s_{i_1}s_{i_2}\cdots s_{i_{k-3}} \left(s_{n-3}s_{n-2}s_{n-3} s_{n-2}\right)s_{i_{k-3}}\cdots s_{i_1}\\
&=&  \beta_p(n-3)^{-1}  \quad \textrm{for some} \quad 1 \le p < n-3\\
&=&  \beta_1(n-3)^{-1} \mod T_n^{''}.
\end{eqnarray*}	    
In all cases, minimality and linear independence of the generating set of $T_n^{'}/T_n^{''}$ shows that $\Theta(a_{i_{k-1}}a_{i_1}a_{i_2}\cdots a_{i_k})$ is not the identity, which proves the claim. 
\par
For the second assertion, since the generators $s_i$ give 2-torsion in  $T_n/T_n^{''}$, it is not Bieberbach. Finally, since $T_n^{'}/T_n^{''}$ is torsion free, it follows that $T_n/T_n^{''}$ has only 2-torsion. 
\end{proof}
\medskip

\subsection{Crystallographic quotients of triplet groups}
As discussed in the proof of Proposition \ref{csp fails for L_n}, $PL_n$ is a free group of finite rank for $n \geq 4$. In fact, one can determine the precise free rank of $PL_n$, which is also implicit in \cite{MR1386845},

\begin{proposition}\label{freeness of pln}
For $n\geq 4$, the pure triplet group $PL_n$ is isomorphic to the free group of rank $1+ n!\Large(\frac{2n-7}{6}\Large)$.
\end{proposition}

\begin{proof}
The permutahedron $\Pi_{n-1}$ is the convex hull of all vectors that are obtained by permuting the coordinates of the vector $(1,2,\ldots,n) \in \mathbb{R}^n$ \cite[Definition 9.6]{MR2237292}. Its vertices can be identified with the elements of $S_n$ in such a way that two vertices are connected by an edge if and only if the corresponding permutations differ by a transposition. Let $\Gamma$ be the 1-skeleton of $\Pi_{n-1}$. Note that $\Gamma$ has $n!$ vertices and $\frac{n!(n-1)}{2}$ edges. Further, there are $\frac{n! (n-2)}{6}$ 6-cycles in $\Gamma$, each of which is uniquely determined by an integer $k \in \{1, 2, \ldots,n-2 \}$ and a permutation $(i_1,i_2,\ldots,i_n)$ with $i_k<i_{k+1}<i_{k+2}$, in which the edge connecting the vertex $(i_1,i_2,\ldots,i_n)$ with the vertex $(i_1, i_2,\ldots,i_{k+1},i_{k},i_{k+2},\ldots, i_n)$ is considered as marked. Let $X_\Gamma$ be a regular 2-dimensional cell complex obtained by attaching a 2-cell along each 6-cycle of $\Gamma$. By the Nerve Theorem \cite[Theorem 10.6]{MR1373690}, the space $Y_n$  is homotopy equivalent to $X_\Gamma$. It is shown in \cite[p.269]{MR1386845} that $X_\Gamma$ is homotopy equivalent to a part of its 1-skeleton, obtained by contracting each 2-cell along with the marked 1-cell onto the remaining boundary of the corresponding 6-cycle. Thus, the resulting graph has $n!$ vertices and $\frac{n!(n-1)}{2} - \frac{n! (n-2)}{6} = \frac{n!(2n-1)}{6}$ edges. It is easy to see that the number of cycles in the resulting 1-skeleton is equal to the number of 4-cycles in $\Gamma$. Using the count of vertices and edges, the Euler characteristic of $X_\Gamma$ is $- n!\big(\frac{2n-7}{6}\big)$. On the other hand, computing the Euler characteristic of $X_\Gamma$ via homology, we conclude that $PL_n$ is the free group of rank $1+ n!\big(\frac{2n-7}{6}\big)$.
\end{proof}

The following table gives rank of $PL_n$ for small values of $n$.
\begin{center}
\begin{tabular}{|c|c|c|c|c|c|}
\hline
$n$ & 3 & 4 & 5 & 6 & 7\\
\hline
rank $(PL_n)$ & 0 & 5 & 61 & 601 & 5881\\
\hline
\end{tabular}
\end{center}

In fact, we can give an explicit free generating set for $PL_4$.

\begin{lemma}\label{pln freeness}
$PL_4 \cong F_5$ and is freely generated by the following elements:
\begin{eqnarray*}
(y_1 y_3)^2 , &~& y_2(y_1 y_3)^2y_2,\\
~ y_1 y_2(y_1 y_3)^2y_2 y_1, &~& y_3y_2(y_1 y_3)^2y_2y_3,\\
~y_1 y_3 y_2(y_1 y_3)^2y_2 y_3 y_1. &~& 
\end{eqnarray*}
\end{lemma}
\begin{proof}
A Schreier set of coset representatives of $PL_4$ in $L_4$ is $$\Lambda_3 \;\cup\; \Lambda_3 y_3 \;\cup \; \Lambda_3 y_3y_2 \;\cup \;\Lambda_3 y_3y_2y_1$$ where $\Lambda_3=\{1, y_1, y_2, y_1y_2, y_2y_1, y_1y_2y_1\}$. Elementary computations using the Reidemeister-Schreier method show that the following generators are non-trivial:
\begin{eqnarray*}
S_{y_1y_3, y_1}=S_{y_3, y_1}^{-1} = S_{y_1y_3y_2y_1,y_2} = S_{y_3y_2y_1,y_2}^{-1} = (y_1 y_3)^2  &:=& x_1,\\ 
S_{y_2y_1y_3, y_1}=S_{y_2y_3, y_1}^{-1} = S_{y_2y_1y_3y_2y_1 ,y_2} = S_{y_2y_3y_2y_1,y_2}^{-1} = y_2((y_1 y_3)^2)y_2 &:=& x_2,\\ 
S_{y_1y_2y_1y_3, y_1}=S_{y_1y_2y_3, y_1}^{-1} = S_{y_1y_2y_1y_3y_2y_1 , y_2} = S_{y_1y_2y_3y_2y_1 , y_2}^{-1}
 =  y_1 y_2((y_1 y_3)^2)y_2 y_1 &:=& x_3 ,\\
 S_{y_3y_2y_1,y_3} = S_{y_2y_3y_2y_1,y_3}^{-1} = y_3y_2((y_1 y_3)^2)y_2y_3 &:=& x_4 ,\\ 
 S_{y_1y_3y_2y_1,y_3} = S_{y_1y_2y_3y_2y_1 , y_3}^{-1} = y_1 y_3 y_2((y_1 y_3)^2)y_2 y_3 y_1 &:=& x_5,\\ 
 S_{y_2y_1y_3y_2y_1 ,y_3} =  S_{y_1y_2y_1y_3y_2y_1 , y_3}^{-1}  = y_2 y_1 y_3 y_2((y_1 y_3)^2)y_2 y_3 y_1 y_2 &:=& x_6.
\end{eqnarray*}
Considering the defining relations, we see that
$$
x_6 = x_2 x_4^{-1} x_1^{-1} x_5 x_3^{-1}.
$$
Hence, $PL_n$ is isomorphic to the free group of rank 5.
\end{proof}

\begin{theorem}\label{ln mod pln cryst}
For $n\geq 4$, there is a short exact sequence 
$$ 1 \longrightarrow PL_n/PL_n^{'} \longrightarrow L_n /PL_n^{'} \stackrel{\pi}{\longrightarrow} S_n \longrightarrow 1 $$
 such that $L_n /PL_n^{'}$ is a crystallographic group of dimension $1+ n!\Large(\frac{2n-7}{6}\Large)$. Furthermore, $L_n /PL_n^{'}$ is not a Bieberbach group.
\end{theorem}

\begin{proof}
Note that $PL_2 = PL_3=1$. Since the commutator subgroup of $PL_n$ is normal in $L_n$, the short exact sequence 
$$ 1 \longrightarrow PL_n \longrightarrow L_n  \stackrel{\pi}{\longrightarrow} S_n \longrightarrow 1 $$
induces the desired sequence. Note that $\Theta(\tau_1\tau_2)((y_1y_3)^2)=y_1y_2\left((y_1y_3)^2\right)y_2y_1 \mod PL_n^{'}$. By Lemma \ref{pln freeness}, both $(y_1y_3)^2$ and $y_1y_2\left((y_1y_3)^2\right)y_2y_1$ lie in a free generating set for $PL_4$, and hence in a free generating set for $PL_n$. Hence, $(y_1y_3)^2\neq y_1y_2\left((y_1y_3)^2\right)y_2 y_1 \mod PL_n^{'}$ and therefore $\tau_1\tau_2 \not\in \ker(\Theta)$. But, for $n>4$, the only proper normal subgroup of $S_n$ that does not contain $\tau_1\tau_2$ is the trivial subgroup, and hence $\Theta$ is faithful for $n>4$. Similarly, for $n=4$, we have $\Theta(\tau_1\tau_3)(y_2\left((y_1y_3)^2\right)y_2)=y_1 y_3 y_2 \left(y_1y_3)^2\right)y_2y_3y_1 \mod PL_4^{'}$. The same reasoning as above shows that $\Theta$ is faithful in this case as well.
\par
Images of $y_i$ and $y_iy_{i+1}$ give order two and order three elements in $L_n /PL_n^{'}$, respectively, and hence it is not a Bieberbach group.
\end{proof}

It follows from \cite[Exercise 9, p.33]{Bourbaki} or a direct application of the Reidemeister--Schreier method that $L_n^{'}$ is a free product of $(n-2)$ copies of $\mathbb{Z}_3$, and hence the abelianisation sequence $1 \to L_n^{'} \to L_n \to \mathbb{Z}_2 \to 1$ does not induce a crystallographic quotient of $L_n$.

\medskip

\begin{ack}
Pravin Kumar is supported by the PMRF fellowship at IISER Mohali. He also thanks NISER Bhubaneswar for the warm hospitality during his visit, where a part of this project was carried out. Mahender Singh is supported by the SwarnaJayanti Fellowship grants DST/SJF/MSA-02/2018-19 and SB/SJF/2019-20/04.
\end{ack}

\section{Declaration}
The authors declare that they have no conflict of interest.

\bibliographystyle{plain}
\bibliography{template}

\medskip

\end{document}